\documentclass{amsart}

\usepackage[normalem]{ulem}

\newtheorem{thm}{Theorem}[section]
\newtheorem*{thm2*}{Theorem 3.2}
\newtheorem*{thm4*}{Theorem 3.4}
\newtheorem*{thm5*}{Theorem 3.5}
\newtheorem{cor}[thm]{Corollary}
\newtheorem{example}[thm]{Example}
\newtheorem{lem}[thm]{Lemma}
\newtheorem{prop}[thm]{Proposition}
\newtheorem{defn}[thm]{Definition}

\newcommand{\m}{{\mathfrak m}}
\newcommand{\p}{{\mathfrak p}}

\newcommand{\N}{{\mathbb N}}

\DeclareMathOperator{\ann}{ann}
\DeclareMathOperator{\Ass}{Ass}
\DeclareMathOperator{\mH}{H}
\DeclareMathOperator{\Hom}{Hom}
\DeclareMathOperator{\supp}{Supp}
\DeclareMathOperator{\soc}{soc}

\makeatletter
\@namedef{subjclassname@2020}{%
  \textup{2020} Mathematics Subject Classification}
\makeatother

\begin{document}


\title{Solution Module and Linear Closure}	

\author{I-Chiau Huang}
\address{Institute of Mathematics, Academia Sinica, Taipei, Taiwan (retired).}
\email{ichiauhuang@gmail.com}

\author{I-Hsun Tsai}
\address{Department of Mathematics, National Taiwan University, Taipei, Taiwan.}
\email{ihtsai@math.ntu.edu.tw}
	
\begin{abstract}
We introduce the notion of an ``initial condition'' for a module $M$ over a 
commutative Noetherian local ring $(A,\m)$, allowing for a recursive construction 
of its ``solution modules''. If $M$ has zero-dimensional support, such as the 
residue field of $A$ and those encountered in residual complexes, we demonstrate 
that the solution module $E(M)$ is an injective hull of $M$. The construction of 
$E(M)$ for finitely generated $M$ is explicit and computable, devoid of the need 
for Zorn's lemma. Further, we improve Baer's criterion for a module $N$ with 
zero-dimensional support to be injective: if any $A$-homomorphism from $\m$ 
to $N$ lifts to $A$, then $N$ is injective. 
\end{abstract}
	
\maketitle


\section{Introduction}


\subsection{Motivation}

Injective modules are a fundamental concept in module theory over a ring. More abstractly, injective objects in an abelian category play a crucial role, particularly to provide homological invariants. Derived categories, introduced by J.-L.~Verdier in his thesis under the guidance of A.~Grothendieck, are associated with abelian categories and were developed to extend Serre’s duality theorem~\cite{ser:cga} to a relative setting~\cite{hart:rd}.

An injective hull of a given module can be understood as a smallest injective module containing it. Defined in terms of a chosen injective hull of the residue 
field of the local ring at each point, residual complexes arise as key objects in 
Grothendieck duality \cite{hart:rd}. The concrete construction of residual complexes is often subtle, due in part to the intricate nature of injective hulls.

In the literature, injective hulls of a given module are often regarded as identical, 
simply because they are isomorphic to each other. However, there is no canonical 
choice among isomorphisms relating two such injective hulls. Moreover, these isomorphisms reveal the rich set-theoretic structure.  For example, such an isomorphism may involve formalisms reminiscent of residues arising in contour 
integration of complex functions. Indeed local duality exhibits such a phenomenon \cite[(7.3)]{hu:pmzds}. More importantly, a concrete description of a residual complex requires distinguishing between different models of injective hulls of a given module.
 
One of the authors in his thesis under the guidance of J.~Lipman constructed injective hulls 
in a relatively canonical  way \cite{hu:pmzds}. Consider a local homomorphism 
$f\colon R\to S$ of Noetherian local rings, which induces a finitely generated residue field 
extension. Let $M$ be an $R$-module with zero-dimensional support, that is, every element 
of $M$ is annihilated by some power of the maximal ideal of $R$. The most important such
modules are injective hulls of the residue field of $R$. An $S$-module $f_\sharp M$ with 
zero-dimensional support was constructed functorially. If $M$ is an injective hull of the residue 
field of $R$, then $f_\sharp M$ is an injective hull of the residue field of $S$. Based on this 
relative construction, an explicit residual complex on a finitely generated algebra was 
canonically constructed from an explicitly given residual complex on its coefficient ring  
\cite{hu:ecrc}. Ideally, it is desirable to construct injective hulls in an absolute context, that is, 
without predefined injective hulls. This provides an impetus for the present paper.

\subsection{Basic Notions}

Injective modules are categorical notion. To be precise, a module $M$ over a commutative ring $A$ 
is called {\em injective} if the functor $\Hom_A(-,M)$ on the category of $A$-modules is exact. 
An $A$-module $E$ containing $M$ is called an {\em injective hull} of $M$ if any one of the 
following equivalent conditions is satisfied.
\begin{itemize}
\item $E$ is a minimal injective module containing $M$.
\item $E$ is a maximal essential extension of $M$.
\item $E$ is an injective module and is an essential extension of $M$.
\end{itemize}
The categorical definition of an injective $A$-module entails a condition that applies to {\em all}
 $A$-modules. It is not unexpected that the equivalence mentioned above relies on Zorn's lemma 
 for its proof. Moreover, the existence of an injective hull of $M$ likewise relies on Zorn's lemma 
 \cite{ec-sch:im}.
 
In this paper, we inspect the role of Zorn's lemma in the theory of injective modules 
and establish a computational foundation. Assume that $A$ is a Noetherian local ring and
$M$ is a module satisfying the initial condition introduced in Definition~\ref{defn:ic}. 
In Definition~\ref{MainConstruction}, we construct a module $E_1(M)$ containing $E_0(M):=M$ parallel to the construction of an algebraic closure of a field. Corollary~\ref{cor:initial} shows that $E_1(M)$ also satisfies the initial condition.
So we can construct $E_2(M)$ from $E_1(M)$ and, in general, construct $E_{i+1}(M)$ from $E_i(M)$ for any $i\geq 0$. Since $E_{i+1}(M)$ provides solutions to linear 
equations in $E_i(M)$, we call the direct limit $E(M)$ of $E_i(M)$ a {\em solution module} of $M$. Along the process, we choose 
a basis for a vector space using Zorn's lemma, which is necessary for infinite-dimensional 
cases but not needed if $M$ is finitely generated. This is the only place where Zorn's lemma 
occurs in our construction. If furthermore $M$ has zero-dimensional support, the solution 
module $E(M)$ turns out to be an injective hull of $M$. This fact requires 
also Zorn's lemma.

The set-theoretic aspect of injective modules appears from systems of linear equations.
Given an ideal $I$ of $A$, an element $\varphi\in\Hom_A(I,M)$ is called an $I$-system, 
since it represents a consistent system $\{aT=\varphi(a)\}_{a\in I}$
of $A$-linear equations with a common unknown $T$.
Solvability of $I$-systems is captured by the exactness of the sequence
\[
\Hom_A(A,M)\to \Hom_A(I,M)\to 0.
\]
The module $M$ is called $I$-closed if the above sequence is exact. Baer's criterion
asserts that $M$ is injective if and only if it is $I$-closed for every ideal $I$ of $A$
\cite{ba:agdsecag}.
The criterion connects categorical and set-theoretic notions via Zorn's lemma.

The table below outlines analogous concepts in the context of fields versus $A$-modules.
\hfill \break
\begin{center}{\begin{tabular}{c||c|c}
    fields & modules & modules \\ 
     & (existing terminology) & (our terminology) \\ \hline
    algebraic  equation & linear equation & linear equation \\
    algebraic  extension & essential extension & linear extension \\
    algebraically closed field & injective module & linearly closed module \\
    algebraic closure & injective hull & linear closure \\
\end{tabular}}\end{center}
\hfill \break
Let $\m$ be the maximal ideal of $A$. Significantly parallel to the first equivalent 
condition for injective hulls, we establish a universal property to prove the following 
result for an $A$-module $M$ that satisfies the initial condition. 
\begin{thm2*}
$E(M)$ is minimal among $\m$-closed modules containing $M$.
\end{thm2*}
\noindent
Exactly parallel to the second and the third equivalent conditions, we prove the following results 
if $M$ furthermore has zero-dimensional support.
\begin{thm4*}
$E(M)$ is a maximal linear extension of $M$.
\end{thm4*}
\begin{thm5*}
$E(M)$ is a linearly closed module and is a linear extension of $M$.
\end{thm5*}

Note that the stated properties in Theorems~\ref{thm:max} and \ref{thm:478987} for $E(M)$ are {\em a priori} different 
from one another.  We prove Theorem~\ref{thm:max} using the theory of associated primes. 
For Theorem~\ref{thm:478987}, our construction ensures that $E(M)$ is a linear extension of $M$.
To prove that $E(M)$ is linearly closed, we have two approaches. One follows 
Theorem~\ref{thm:max}. Another using Artin-Rees lemma is constructive. Of course, 
if we employ Zorn's Lemma, these properties for $E(M)$ become equivalent without invoking the 
theory of commutative Noetherian rings.

\subsection{Related Works}

The classical result of Matlis \cite{mat:imnr} asserts that an injective hull $E'$ of $A/\m$ 
is a direct limit of $E'_i=\Hom_A(A/\m^i,E')$ considered as submodules of $E'$, 
where the dimension of the vector space $E'_{i+1}/E'_i$ over $A/\m$ is computed. 
See also \cite[Proposition 7.5, Chapter II]{hart:rd}.
These $E'_i$ can be identified with solution modules $E_i(A/\m)$. While Matlis 
presupposes an injective hull $E'$ to prove his result, we construct $E_i(A/\m)$ 
from scratch and produce an injective hull $E(A/\m)$ of $A/\m$.

The main examples of explicit injective hulls all originate from inverse systems of polynomials 
over a field, a concept developed 
by Macaulay \cite{mac:atms}. See also Kucera \cite{kuc:edie}. The relative canonical injective hulls 
given in \cite{hu:pmzds} combine two special constructions, both of which relate to inverse systems: 
one by local cohomology under certain smooth condition and the other by continuous 
homomorphisms under certain finiteness condition. In particular, elements of these local cohomology
modules are described using generalized fractions, whose denominators resemble inverse monomials.
Explicit constructions of injective hulls in the literature all require extra conditions.
Look at a prime ideal $\p$ of positive hight of a commutative Noetherian ring $A$. 
Pournaki and Tousi \cite{po-to:edciim} construct explicitly an injective hull of $A/\p$ requiring that 
$\p A_\p$ is principal. In this paper, our explicit construction 
provides an injective hull for the residue field of an arbitrary Noetherian local ring with 
no additional conditions imposed.

Let $I$ be an ideal of $A$. An $I$-system in an injective module $E$ can be always solved.
However, it is not found in the literature how an explicitly solution can be found.
The second proof of Theorem~\ref{thm:478987} provides a constructive method to solve $I$-systems in an injective hull. In Proposition~\ref{prop:9yyyq4}, we work out certain systems of equations for a power series ring over a field, in which inverse polynomials occur as solutions. 

It is well known that every injective $A$-module decomposes as a direct sum of 
indecomposable injective modules. Each indecomposable injective module serves as an 
injective hull of the residue field of the local ring $A_\p$ for some prime ideal $\p$ 
of $A$. In this context, the structure of injective modules is elucidated by the structure of injective hulls of the residue field of each local ring. The breadth of applications of our construction of linear closure also becomes evident, since the residue field has zero-dimensional support and satisfies the initial condition, and with Zorn's lemma, the linear closure is revealed to be an injective hull. As an application of our construction, Theorem~\ref{cor:BC} improves Baer's criterion for modules with zero-dimensional support.

\subsection{Constructions and main results}

Starting from a module $M$ satisfying the initial condition (Definition~\ref{defn:ic}), 
we construct a linear extension $E_1(M)$ of $M$ (Definition~\ref{MainConstruction} and
Proposition~\ref{prop:85439}). 
Since $E_1(M)$ also satisfies the initial condition
(Corollary~\ref{cor:initial}), we can construct $E_2(M)$ from $E_1(M)$. In
general, we can construct $E_{i+1}(M)$ from $E_i(M)$ for $i\geq 1$.
We define the solution module $E(M)$ of $M$ as the direct limit of $E_i(M)$ (Definition~\ref{def:sm}).


\section{Solution module}\label{sec:sm}


Throughout this paper, $A$ is a commutative ring with an identity element. 
The structure of an $A$-module
$M$ can be understood through system
\begin{equation}\label{eq:closure}
\{a_{i1}T_1+\cdots+a_{i\ell}T_\ell=\alpha_i\}_{i\in\Gamma}
\end{equation}
of $A$-linear equations together with its solutions, where  $T_1,T_2,\ldots,T_\ell$ are unknowns,
$a_{ij}\in A$, $\alpha_i\in M$ and $\Gamma$ is an index set. A necessary
condition for (\ref{eq:closure}) having a solution is
\begin{eqnarray*}
b_1a_{i_1j}+\cdots+b_\ell a_{i_\ell j}=0 \text{ for $1\leq j\leq\ell$}& \implies &
b_1\alpha_{i_1}+\cdots+b_\ell \alpha_{i_\ell}=0,
\end{eqnarray*}
where $i_j\in\Gamma$ and $b_i\in A$. Such a condition on the module $M$ is called a 
{\em constraint} from the underlying ring $A$. 

A system of $A$-linear equations in $M$ is called {\em consistent}, if it satisfies all constraints from $A$. 
A consistent system of linear equations is simply called a {\em consistent system}.
For the case $\ell=1$, the system of equations (\ref{eq:closure}) is consistent if 
and only if $\alpha_i=f(a_{i1})$ for some $A$-linear map $f:I\to M$, where $I$ is 
the ideal generated by $a_{i1}$. Referring to an ideal $I$ of $A$, we call such 
a consistent system an $I$-system in $M$ and identify it with an element in 
$\Hom_A(I,M)$. A solution to an $I$-system $f$ in $M$ is an element $\alpha\in M$ 
such that $f(a)=a\alpha$ for all $a\in I$. Multiplying an element in $M$ by an 
element in $I$ gives rise to an $A$-linear map $M\to\Hom_A(I,M)$. The $A$-linear 
map is surjective if and only if every $I$-system in $M$ has a
solution. 

We rename a classical notion to exhibit its analogy with algebraic extension in the theory of fields.
\begin{defn}[Linear extension]
An $A$-module $N$ containing $M$ is called a {\em linear extension} of $M$, if 
every element of $N$ is a solution to a non-zero $I$-system in $M$. For a linear 
extension $N$ of $M$, we also say that $N$ is linear over $M$. 
\end{defn}
Clearly, $N$ is linear over $M$ if and only if, for every non-zero element $\alpha\in N$, 
there exists an element $a\in A$ such that $0\neq a\alpha\in M$. In the literature, the latter is taken as 
the definition of an essential extension. Assuming furthermore that $A$ 
is a Noetherian local ring with the maximal ideal $\m$ and the residue field $\kappa$. 
The {\em socle}  
of an $A$-module $N$, denoted by $\soc N$, is $\{\alpha\in N\colon \m\alpha=0\}$, which is the largest $A$-submodule naturally inheriting a structure as a $\kappa$-vector space.
\begin{lem}\label{lem:87643}
If $N$ is linear over $M$, then $\soc N=\soc M$.
\end{lem}
\begin{proof}
Clearly $\soc M\subset \soc N$, since $M$ is a submodule of $N$. To show 
$\soc N\subset \soc M$, we consider $0\neq\alpha\in\soc N$. There exists an element $a\in A$ such that
$0\neq a\alpha\in M$. Since $\m\alpha=0$, the element $a$ has to be invertible.
Therefore $\alpha=a^{-1}(a\alpha)\in M$ and hence $\alpha\in\soc M$.
\end{proof}

Starting with a module with the following condition, 
we would like to join solutions to $\m$-systems in the module.
\begin{defn}[Initial condition] \label{defn:ic}
We say that an $A$-module $M$ satisfies the 
initial condition if every $\m$-system in $\m M$ has a solution in $M$.
\end{defn}
The initial condition for $M$ is equivalent to the exactness of the sequence
\[
0\to\soc M\to M\to\Hom_A(\m,\m M)\to 0. 
\]
The condition $\soc M=M$ is equivalent to $\m M=0$. In such a case, the $A$-module 
structure on $M$ induces a natural vector space structure on $M$ over the residue 
field $\kappa$ of $A$; we refer to $M$ as a $\kappa$-vector space. The initial condition 
is satisfied if $M$ is a $\kappa$-vector space, for example, if $M=\kappa$. 
The initial condition does not always hold.
\begin{example}\label{eq:p1yh}
Consider the local ring $A=\kappa[X,Y]/\langle X^2,XY,Y^2\rangle$, where $\kappa$ is
a field. The maximal ideal $\m$ is generated by the images $x$ and $y$ of $X$ and $Y$. 
The $\m$-system in $\m$ given by $x\mapsto y$ and $y\mapsto x$ does not have a solution
in $A$.  Therefore $A$ as a module over itself does not satisfy the initial condition.
\end{example}

The $A$-module $\Hom_A(\m,M)/\Hom_A(\m,\m M)$ is a $\kappa$-vector space. 
Indeed, it is a subspace of the $\kappa$-vector space $\Hom_A(\m,M/\m M)$. 
\begin{defn}[Solution module up to the first order]\label{MainConstruction}\label{def:76167}
Let $A$ be a Noetherian local ring with the maximal ideal $\m$ and the residue field $\kappa$. 
Let $M$ be an $A$-module satisfying the initial condition. We choose an index set 
$\Gamma$ and $\{f_i\}_{i\in\Gamma}\subset\Hom_A(\m,M)$ whose images in the vector
space
$\Hom_A(\m,M)/\Hom_A(\m,\m M)$ form a $\kappa$-basis. Let $F$ be the free 
$A$-module with a basis $\{e_i\}_{i\in\Gamma}$. Let $N$ be the submodule of 
$F\oplus M$ generated by $(be_i,-f_i(b))$, where $b\in\m$ and $i\in\Gamma$. 
Depending on a choice of a $\kappa$-basis, we define $E_1(M):=(F\oplus M)/N$ and
call it a solution module of $M$ up to the first order. 
\end{defn}

If the vector space $\Hom_A(\m,M)/\Hom_A(\m,\m M)$ is zero, we write $E_1(M):=M$ 
as a convention. Since $(0\oplus M)\cap N=0$, we may regard $M$ as a submodule of 
$E_1(M)$. In Proposition~\ref{prop:9yyyq4}, we will see how inverse polynomials occur in solution modules. The negative exponents of monomials justify our choice of terminology regarding the order of solution modules. We remark that the existence of a basis of an 
arbitrary vector space uses Zorn's lemma. If $M$
is finitely generated, the construction of $E_1(M)$ does not involve Zorn's lemma.
\begin{prop}\label{prop:solution}
Every $\m$-system in $M$ has a solution in $E_1(M)$.
\end{prop}
\begin{proof}
Given an $\m$-system $f$ in $M$, we choose $a_1,\ldots,a_n\in A$ and 
$i_1,\ldots,i_n\in\Gamma$ such that $f-(a_1f_{i_1}+\cdots+a_nf_{i_n})$ is an 
$\m$-system in $\m M$. Let $\alpha\in M$ be a solution to the $\m$-system
$f-(a_1f_{i_1}+\cdots+a_nf_{i_n})$ in $\m M$. Then $\alpha+(a_1e_{i_1}+\cdots+a_ne_{i_n})\in E(M)$ is a 
solution to the $\m$-system $f$. Indeed,
\[
b(\alpha+a_1e_{i_1}+\cdots+a_ne_{i_n})=b\alpha+a_1f_{i_1}(b)+\cdots+a_nf_{i_n}(b)=f(b)
\]
for any $b\in\m$.
\end{proof}

\begin{cor}\label{cor:initial}
$E_1(M)$ satisfies the initial condition.
\end{cor}
\begin{proof}
Since $\m E_1(M)\subset  M$, every $\m$-system in 
$\m E_1(M)$ is also an $\m$-system in $M$. So it has a solution in $E_1(M)$. 
\end{proof}

Given $(a_1e_{i_1}+\cdots+a_ne_{i_n},\alpha)\in F\oplus M$, we denote its image in 
$E_1(M)$ by $a_1e_{i_1}+\cdots+a_ne_{i_n}+\alpha$. For instance, we may write 
$be_i=f_i(b)$ for any $b\in\m$ and $i\in\Gamma$. Therefore 
$\m E_1(M)\subset M$ from the defining relation of $N$. 
Since the image of $\{f_i\}_{i\in\Gamma}$ forms a $\kappa$-vector space basis, 
we may prove that the following conditions are equivalent for 
$a_1,\ldots,a_n\in A$ and $i_1,\ldots,i_n\in\Gamma$.
\begin{itemize}
\item $a_1,\ldots,a_n\in\m$
\item $a_1e_{i_1}+\cdots+a_ne_{i_n}\in M$
\item $a_1f_{i_1}+\cdots+a_nf_{i_n}\in \Hom_A(\m,\m M)$
\end{itemize}
The $A$-linear map $F\oplus M\to \Hom_A(\m,M)$ defined by 
\[
(a_1e_{i_1}+\cdots+a_ne_{i_n},\alpha)\mapsto a_1f_{i_1}+\cdots+a_nf_{i_n}
\]
gives rise to an exact sequence
\[
0\to M\to E_1(M)\to \Hom_A(\m,M)/\Hom_A(\m,\m M)\to 0.
\]
So what we construct is an extension of $\Hom_A(\m,M)/\Hom_A(\m,\m M)$ by $M$. 

\begin{prop}\label{prop:85439}
$E_1(M)$ is a linear extension of $M$.
\end{prop}
\begin{proof}\mbox{}
Consider $\omega=a_1e_{i_1}+\cdots+a_ne_{i_n}+\alpha\in E_1(M)\setminus M$, where
$a_1,\ldots,a_n\in A$, $i_1,\ldots,i_n\in\Gamma$ and $\alpha\in M$. Since 
$a_1f_{i_1}+\cdots+a_nf_{i_n}\not\in\Hom_A(\m,\m M)$, there exists $b\in\m$ such that
$(a_1f_{i_1}+\cdots+a_nf_{i_n})(b)\in M\setminus \m M$. Then
$b\omega\in M\setminus \m M$. In particular,
$b\omega\neq 0$. 
\end{proof}
The following fact follows directly from Lemma~\ref{lem:87643}.
\begin{cor}\label{cor:985678}
$\soc E_1(M)=\soc M$.
\end{cor}

Different choices of $\kappa$-bases give isomorphic solution modules up to the first order:
We may also choose another subset $\{f'_i\}_{i\in\Gamma}$ of $\Hom_A(\m,M)$, whose images in 
$\Hom_A(\m,M)/\Hom_A(\m,\m M)$ form a $\kappa$-basis. 
Let $F'$ be the free $A$-module with a basis $\{e'_i\}_{i\in\Gamma}$.
Let $E_1'(M)$ be the quotient module of 
$F'\oplus M$ modulo the submodule generated by $(be'_i,-f'_i(b))$, where $b\in\m$ and $i\in\Gamma$. 
\begin{prop}\label{prop:66772}
There is an isomorphism $E_1(M)\to E_1'(M)$ stabilizing $M$. 
\end{prop}
\begin{proof}
For $i\in\Gamma$, there exist $g_i\in\Hom_A(\m,\m M)$ and finitely many $i_1,i_2,\ldots\in\Gamma$ and $a_{i_1},a_{i_2},\ldots\in A$ such that 
$f'_i=g_i+a_{i_1}f_{i_1}+a_{i_2}f_{i_2}+\cdots$. With the initial condition, there 
exists $\alpha_i\in M$ 
such that $g_i(b)=b\alpha_i$ for $b\in\m$. The $A$-linear map
$\Phi'\colon E_1'(M)\to E_1(M)$ stablizing $M$ and given by 
$e'_i\mapsto \alpha_i+a_{i_1}e_{i_1}+a_{i_2}e_{i_2}+\cdots$ is well-defined. 
Similarly, we have an $A$-linear map $\Phi\colon E_1(M)\to E_1'(M)$ stabilizing $M$. 
The composition $(\Phi'\circ\Phi)\colon E_1(M)\to E_1(M)$ stabilizes $M$ and satisfies 
$\alpha-(\Phi'\circ\Phi)(\alpha)\in M$ for any $\alpha\in E_1(M)$. If $(\Phi'\circ\Phi)(\alpha)=0$,
then $\alpha\in M$ and hence $\alpha=(\Phi'\circ\Phi)(\alpha)=0$. In other words, 
$\Phi'\circ\Phi$ is one-to-one. It implies that $\Phi$ is one-to-one as well. 

To show $\Phi$ is onto, we consider $\alpha'\in E_1'(M)$, which gives rise to an $\m$-system
in $M$ by $b\mapsto b\alpha'$. Let $\alpha\in E_1(M)$ be a solution to the system. Then
$\beta:=\alpha'-\Phi(\alpha)\in\soc E_1(M)$. By Corollary~\ref{cor:985678}, $\beta\in M$.
Therefore $\alpha'=\Phi(\alpha+\beta)$, which means $\Phi$ is onto.
\end{proof}

While providing solutions to $\m$-systems in $M$, the module $E_1(M)$ does
not create new solutions to $\m$-systems in $\m M$.
\begin{prop}
If $\omega\in E_1(M)$ is a solution to an $\m$-system in $\m M$, 
then $\omega\in M$.
\end{prop}
\begin{proof}
We suppose that $\omega=a_1e_{i_1}+\cdots+a_ne_{i_n}+\alpha$ is a solution to
$f\in\Hom_A(\m,\m M)$, where $a_1,\ldots,a_n\in A$, $i_1,\ldots,i_n\in\Gamma$ and 
$\alpha\in M$. For $b\in\m$,
\[
a_1f_{i_1}(b)+\cdots+a_nf_{i_n}(b)=b(a_1e_{i_1}+\cdots+a_ne_{i_n})=f(b)-b\alpha\in \m M.
\]
It follows that $a_1f_{i_1}+\cdots+a_nf_{i_n}\in\Hom_A(\m,\m M)$ and $a_1,\ldots,a_n\in\m$.
Therefore 
\[
\omega=a_1e_{i_1}+\cdots+a_ne_{i_n}+\alpha=f_{i_1}(a_1)+\cdots+f_{i_n}(a_n)+\alpha\in M.
\]
\end{proof}

The $A$-module $E_1(M)$ is a smallest module containing $M$, that provides
solutions to $\m$-systems in $M$.
This can be stated as a universal property:
\begin{prop}\label{prop:8577}
Let $E'$ be an $A$-module containing $M$. If every $\m$-system in $M$ has a 
solution in $E'$, there exists a one-to-one $A$-linear map $E_1(M)\to E'$ stablizing $M$.
\end{prop}
\begin{proof}
For $i\in\Gamma$, let $\beta_i\in E'$ be a solution to $f_i$, that is, $f_i(b)=b\beta_i$ 
for $b\in\m$. With the notation in Definition~\ref{MainConstruction}, the map $F\to E'$ given by 
$e_i\mapsto\beta_i$ induces an $A$-linear map $E_1(M)\to E'$ stablizing $M$. 
Consider an element $a_1e_{i_1}+\cdots+a_ne_{i_n}+\alpha$ in the kernel of 
the induced map, where $a_1,\ldots,a_n\in A$, $i_1,\ldots,i_n\in\Gamma$ and 
$\alpha\in M$. For $b\in\m$,
\[
a_1f_{i_1}(b)+\cdots+a_nf_{i_n}(b)=b(a_1\beta_{i_1}+\cdots+a_n\beta_{i_n})=-b\alpha\in\m M.
\]
In other words, $a_1f_{i_1}+\cdots+a_nf_{i_n}\in \Hom_A(\m,\m M)$. Therefore 
$a_1,\ldots,a_n\in\m$. The element
$a_1e_{i_1}+\cdots+a_ne_{i_n}n+\alpha=f_{i_1}(a_1)+\cdots+f_{i_n}(a_n)+\alpha\in M$
in the kernel is stable under the map $E_1(M)\to E'$ and hence vanishes.
In other words, the map $E_1(M)\to E'$ is one-to-one.
\end{proof}

By Corollary~\ref{cor:initial}, we may repeat the construction of solution modules. 
Starting from $E_0(M):=M$, we define $E_i(M):=E_1(E_{i-1}(M))$ for $i\geq 1$ and
call $E_i(M)$ the solution module of $M$ up to the $i$-th order.
We obtain a filtration
\[
0\subset E_0(M)\subset E_1(M)\subset E_2(M)\subset\cdots
\]
of linear extensions of $M$ such that, for $i\geq 0$, $\m E_{i+1}\subset E_i$ and every $\m$-system 
in $E_i(M)$ has a solution in $E_{i+1}(M)$. In other words, there is a well-defined 
surjective $A$-linear map $E_{i+1}(M)\to\Hom_A(\m,E_i(M))$.
We have an exact sequence
\begin{equation}\label{eq:781677}
0\to E_i(M)\to E_{i+1}(M)\to\Hom_A(\m,E_i(M))/\Hom_A(\m,E_{i-1}(M))\to 0.
\end{equation}
for $i\geq 1$.

\begin{prop}\label{prop:50163}
Every $\m^i$-system in $E_j(M)$ has a solution in $E_{i+j}(M)$.
\end{prop}
\begin{proof}
We prove the lemma by induction on $i$. The case $i=1$ is Proposition~\ref{prop:solution}. 
Assume that $i>1$ and the proposition holds for any $j$ and for powers of $\m$ less than $i$. 
Consider a consistent system $\varphi\colon\m^i\to E_j(M)$. Let $a_1,\ldots,a_n$ be minimal 
generators for $\m^{i-1}$. Fix a natural number $\ell$ less than or equal to $n$. The $\m$-system 
$b\mapsto\varphi(ba_\ell)$ has a solution $\alpha_\ell\in E_{j+1}(M)$, that is, 
$\varphi(ba_\ell)=b\alpha_\ell$ for $b\in\m$. If $\sum c_\ell a_\ell=0$
for $c_1,\ldots,c_n\in A$, then $c_\ell\in\m$ since $a_1,\ldots,a_n$ are minimal generators.
Hence $\sum c_\ell\alpha_\ell=\varphi(\sum c_\ell a_\ell)=0$. Therefore we have an 
$\m^{i-1}$-system in $E_{j+1}(M)$ given by $a_\ell\mapsto\alpha_\ell$ for $\ell=1,\ldots,n$.
The hypothesis of our induction gives a solution 
$\alpha\in E_{i+j}(M)$ to the $\m^{i-1}$-system, that is, $\alpha_\ell=a_\ell\alpha$ for all $i$. Clearly $\alpha$ is also 
a solution to the $\m^i$-system $\varphi$.
\end{proof}

\begin{defn}[Solution module]\label{def:sm}
The direct limit $E(M)$ of $\{E_i(M)\}_{i\geq 0}$  is called the solution module of $M$.
\end{defn}
Solution module $E(M)$ is a linear extension of $M$. Regarding $E_i(M)$ as a subset of $E(M)$, 
we have $E(M)=\cup_{i\geq 0}E_i(M)$. 


\section{Linear closure}\label{sec:lc}


\begin{defn}[Linear closure]
For an ideal $I$ of $A$, an $A$-module $M$ is called $I$-closed if
every $I$-system in $M$ has a solution. The module
$M$ is called {\em linearly closed} if it is $I$-closed for any ideal $I$ of $A$. A linear extension of $M$ is called a {\em linear closure} of $M$ if the extension 
is linearly closed. 
\end{defn}
With Zorn's lemma, linearly closed modules and injective modules become the same notion. 
This is Baer's criterion for a module to be injective. 
\begin{thm}\label{thm:5277}
Let $A$ be a Noetherian local ring with the maximal ideal $\m$, and let $M$ be an 
$A$-module satisfying the initial condition. Then $E(M)$ is minimal among $\m$-closed
modules containing $M$.
\end{thm}
\begin{proof}
Since $\m$ is finitely generated, an $\m$-system in $E(M)$ is actually an $\m$-system in certain $E_i(M)$
and hence has a solution in $E_{i+1}(M)$. If $M$ is a submodule of an $\m$-closed 
$A$-module $E'$, repeatedly applying Proposition~\ref{prop:8577}, 
we embed $E(M)$ into $E'$.
\end{proof}

To investigate maximal linear extensions, we use the theory of associated primes 
for Noetherian rings \cite[Chapter IV]{bo:ca1-7}, where the modules are not required 
to be finitely generated, cf. \cite[Theorem 6.5]{mats:crt}.
Let $A$ be a Noetherian local 
ring with the maximal ideal $\m$. The {\em support} of an $A$-module $N$ is the 
set consisting of those prime ideals 
$\p$ such that the localization $N_\p$ does not vanish. An {\em associated prime} 
of an $A$-module $N$ is a prime ideal of the form 
$\ann\alpha:=\{a\in A\colon a\alpha=0\}$ for some $\alpha\in N$. The support
and the set of associated prime of $N$ are denoted by $\supp(N)$ and $\Ass(N)$, 
respectively. 
It holds true that $\Ass(N)\subset\supp(N)$. Furthermore, the set
of minimal elements of $\Ass(N)$ and of $\supp(N)$ coincide.
If $N\neq 0$, then both $\supp(N)$ and $\Ass(N)$ are not empty. 
If an $A$-module $N'$ is a linear extension of $N$, then $\Ass(N')=\Ass(N)$. 

Following the terminology due to J.~Lipman \cite{hu:pmzds}, we say that an 
$A$-module $N$ has zero-dimensional support if every element of $N$ is 
annihilated by some power of $\m$. If $N$ has zero-dimensional support, 
then $\ann\alpha\not\subset\p$ for any non-zero $\alpha\in N$ and prime 
ideal $\p$ distinct from $\m$. Hence $N_\p=0$ and $\supp(N)=\{\m\}$. 
Conversely, if $\supp(N)=\{\m\}$, then 
$\Ass(A/\ann\alpha)=\Ass(A\alpha)=\supp(A\alpha)=\{\m\}$ for any non-zero 
element $\alpha\in N$. In other words, $\ann\alpha$ is a primary ideal belong to $\m$. Hence the element $\alpha$ is annihilated 
by some power of $\m$. This justifies the use of the term ``zero-dimensional support''. 
The module in Example~\ref{eq:p1yh} has zero-dimensional support but does not satisfy 
the initial condition. On the other hand, the following example satisfies the initial condition
but with the support not zero-dimensional. 

\begin{example}
Let $A$ be the power series ring $\kappa[\![X]\!]$ over a field $\kappa$. The support of 
its quotient field $\kappa(\!(X)\!)$ consists of the zero ideal besides the maximal ideal. As an 
$A$-module, $\kappa(\!(X)\!)$ is linearly closed. In particular, it satisfies the initial condition.
\end{example}

\begin{thm}\label{thm:max}
Let $A$ be a Noetherian local ring with the maximal ideal $\m$, and let $M$ be an $A$-module that satisfies 
the initial condition. If $M$ has zero-dimensional support, then $E(M)$ is a maximal 
linear extension of $M$.
\end{thm}
\begin{proof}
Let $\alpha$ be a non-zero element of a linear extension $E'$ of $E(M)$. Since 
\[
\emptyset\neq\Ass(A/\ann\alpha)\subset\Ass(E')=\Ass(E(M))=\Ass(M)=\{\m\},
\] 
the maximal ideal $\m$ is the only associated prime of $A/\ann\alpha$. The maximal
ideal $\m$ is a minimal elements in $\supp(A/\ann\alpha)$. Hence $A/\ann\alpha$ 
has zero-dimensional support. The image of the identity of $A$ in $A/\ann\alpha$ is
annihilated by a power of $\m$. Equivalently, $\alpha$ is annihilated by some power of $\m$. We use induction 
on $i$ to prove the claim that 
$\m^i\alpha=0 \implies\alpha\in E_{i-1}(M)$.
The case $i=1$ is Lemma~\ref{lem:87643}. Now we assume $i>1$ and the claim holds for exponents
of $\m$ smaller than $i$. 
Fix an element $a\in\m$. Since $\m^{i-1}(a\alpha)=0$, by the induction hypothesis
$a\alpha\in E_{i-2}(M)$. The consistent system $\{aT=a\alpha\}_{a\in\m}$ in $E_{i-2}(M)$ has 
a solution $\beta\in E_{i-1}(M)$. Since $\m(\alpha-\beta)=0$, we conclude that 
$\alpha-\beta\in E_0(M)$ and $\alpha\in E_0(M)+E_{i-1}(M)=E_{i-1}(M)$.
\end{proof}

\begin{thm}\label{thm:478987}
Let $A$ be a Noetherian local ring with the maximal ideal $\m$, and let $M$ be an 
$A$-module that satisfies the initial condition. If $M$ has zero-dimensional support, 
then $E(M)$ is a linear closure of $M$. 
\end{thm}
\begin{proof}
By Proposition~\ref{prop:85439}, $E(M)$ is a linear extension of $M$. To show that 
$E(M)$ is linearly closed, we need to find a solution to any given consistent system 
$\varphi\colon I\to E(M)$ for any proper ideal $I$ of $A$. By Proposition~\ref{prop:50163},
it suffices to extend $\varphi$ from $I$ to the maximal ideal $\m$ of $A$. For an element
$a\in\m\setminus I$, we would like to extend $\varphi$ from $I$ to $I+Aa$.
This is equivalent to find an element $\alpha\in E(M)$ such that
\begin{equation}\label{eq:99511}
b\in I\colon a \implies\varphi(ba)=b\alpha.
\end{equation}
With such an element $\alpha$, the map $\varphi$ can be extended to $I+Aa$ by setting 
$c+ba\mapsto\varphi(c)+b\alpha$ for $c\in I$ and $b\in A$.
By the ascending chain condition of the Noetherian ring, we can assume 
that there are no elements $a\in\m\setminus I$ and $\alpha\in E(M)$ satisfying the
condition (\ref{eq:99511}). Under the assumption, we claim that $I=\m$. 
 
Let $N$ be the submodule of $\m\oplus E(M)$ consisting of $(a,-\varphi(a))$, 
where $a\in I$. Through the canonical maps
$E(M)\to\m\oplus E(M)\to(\m\oplus E(M))/N$,
we may identify $E(M)$ with the submodule $(I\oplus E(M))/N$ of 
$(\m\oplus E(M))/N$. If $I\neq\m$, we choose an element $a\in\m\setminus I$.
There exists $b\in A$ such
that $ba\in I$ but $\varphi(b\alpha)\neq b\alpha$ for any $\alpha\in E(M)$.
Then $b(a,\alpha)\not\in N$ and
its image is in $E(M)$. This shows that $(\m\oplus E(M))/N$ is linear over $E(M)$.
By Theorem~\ref{thm:max}, $(\m\oplus E(M))/N=E(M)$, equivalently
$\m\oplus E(M)=I\oplus E(M)$. We arrive at the 
contradiction that $I=\m$.
\end{proof}
We remark that we do not have an algorithm to find an element $\alpha\in E(M)$ satisfying 
the condition (\ref{eq:99511}). Using Artin-Rees lemma, we provide yet another proof 
that $E(M)$ is linearly closed. Let $I$ be an ideal of $A$. Artin-Rees lemma asserts the 
existence of $k\in\N$ such that $\m^\ell\cap I=\m^{\ell-k}(\m^k\cap I)$ for every
$\ell\geq k$. It implies $\m^\ell\cap I\subset\m^{\ell-k}I$.
\begin{proof} 
We need to find a solution to $\varphi\colon I\to E(M)$. 
The image of $\varphi$ is annihilated by $\m^{\ell_1}$ for some $\ell_1\in\N$. Therefore 
$\m^{\ell_1} I\subset\ker\varphi$. By Artin-Rees lemma, 
$\m^{\ell_2}\cap I\subset\ker\varphi$ for some $\ell_2\geq\ell_1$. 
Mapping every element in $\m^{\ell_2}$ to zero, we can extend 
$\varphi$ to $I+\m^{\ell_2}$. So we may assume that $\m^{n+1}\subset I$ for some
$n\in\N$. Assume that there exists $b\in\m^n\setminus I$.
Let $\alpha\in E(M)$ be a solution to the $\m$-system $a\mapsto\varphi(ab)$,
that is, $\varphi(ab)=a\alpha$ for $a\in\m$.
Note that, if $c_1+a_1b=c_2+a_2b$ for $c_1,c_2\in I$ and $a_1,a_2\in A$, then
$(a_1-a_2)b=c_2-c_1\in I$. Since $b\not\in I$, the element $a_1-a_2$ cannot
be invertible, equivalently, $a_1-a_2\in\m$. Therefore
\[
\varphi(c_1)+a_1\alpha=
\varphi(c_2)+a_2\alpha+\varphi\left((a_2-a_1)b\right)-(a_2-a_1)\alpha=
\varphi(c_2)+a_2\alpha.
\]
So we may extend $\varphi$ to $I+Ab$ by $c+ab\mapsto\varphi(c)+a\alpha$ for $a\in I$ 
and $b\in\m$. If $\m^n\not\subset I+Ab$, we may further extend $\varphi$. By the ascending 
chain condition of Noetherian rings, $\varphi$ can be extended to an ideal containing $\m^n$. Keep on going,
$\varphi$ can be extended to an ideal containing $\m^{n-1}$, $\m^{n-2}$ and so on. Eventually $\varphi$ is extended to $\m$, over which it has a solution.
\end{proof}
If $M$ is finitely generated, the second proof provides a constructive method to solve the 
$I$-system $\varphi$. First we compute $I\cap\m^i$ for $i=1,2,\ldots$ to find an $n\in\N$ such 
that $\varphi(I\cap\m^n)=0$. Then we extend $\varphi$ from $I$ to $I+\m^n$ as in the proof 
so that we may assume $\m^n\subset I$ for some $n\in\N$. The next step is a membership 
problem to find an $n\in\N$ such that $\m^n\subset I$ but $\m^{n-1}\not\subset I$. It remains 
to solve some concrete $\m$-systems in $E(M)$, which can be work out within finite dimensional vector
spaces.

As an application of our construction, we may improve Baer's criterion for a module 
with zero-dimensional support.
\begin{cor}\label{cor:BC}
Let $M$ be an $A$-module with zero-dimensional support. If $M$ is $\m$-closed, 
it is linearly closed.
\end{cor}
\begin{proof}
The module $M$ satisfies the initial condition, so we may construct $E_1(M)$.
As an $\m$-closed module, $M$ does not introduce new elements to its solution module, 
that is, $M=E_1(M)=E(M)$. Since $M$ has zero-dimensional support, $M$ is linearly
closed by Theorem~\ref{thm:478987}.
\end{proof}

Our construction of solution module does not necessarily preserves vector space: Begin with 
an $A$-module $M$ satisfying $\m M=0$, it is not generally the case that $\m E(M)=0$. 
However, our construction  
preserves socles by Lemma~\ref{lem:87643}. Furthermore, modules with zero-dimensional 
support are also preserved.
\begin{prop}
Let $M$ be an $A$-module satisfying the initial condition.
If $M$ has zero-dimensional support, so does $E(M)$.
\end{prop}
\begin{proof}
Consider $\omega\in E_1(M)$. 
Assume that $\m$ is generated by $a_1,\ldots,a_m$. Since $a_j\omega\in M$, there exits $n\in\N$
such that $\m^n a_j\omega=0$ uniformly for $1\leq j\leq m$. Then $\m^{n+1}\omega=0$. Therefore
$E_1(M)$ has zero-dimensional support, as do every $E_i(M)$ and $E(M)$.
\end{proof}

We provide examples to illustrate our construction of explicit linear closures.
For a Noetherian local ring $A$ containing a field $\kappa$ with the maximal
ideal $\m$, we consider the $A$-module $\Hom_\kappa^c(A,\kappa)$ 
consisting of $\kappa$-linear maps $\varphi\colon A\to\kappa$ continuous in the sense that $\varphi(\m^i)=0$ for some $i$. If the residue field of $A$ is finitely generated as a $\kappa$-vector space, then $\Hom_\kappa^c(A,\kappa)$ is an injective hull of the residue field of $A$; see \cite[(3.4)]{hu:pmzds} for a  general statement in the relative situation. 

The special case $A=\kappa[\![x_1,\ldots,x_n]\!]$ explains the idea of Macaulay 
\cite{mac:atms}. For $i_1,\ldots,i_n>0$, we denote by
$x_1^{-i_1}\cdots x_n^{-i_n}$ the continuous $\kappa$-linear map $A\to\kappa$ sending the 
monomial $x_1^{i_1-1}\cdots x_n^{i_n-1}$ to $1$ and sending other monomials to zero.
These continuous $\kappa$-linear maps, called inverse monomials,
form a basis for $\Hom_\kappa^c(A,\kappa)$ as a $\kappa$-vector space. 
We call elements of $\Hom_\kappa^c(A,\kappa)$ inverse polynomials.
The module structure of $\Hom_\kappa^c(A,\kappa)$ is described by
\[
(x_1^{j_1}\cdots x_n^{j_n})(x_1^{-i_1}\cdots x_n^{-i_n})=
\begin{cases}
x_1^{j_1-i_1}\cdots x_n^{j_n-i_n},&\text{if $i_1>j_1,\ldots,i_n>j_n$;}\\
0,&\text{otherwise.}
\end{cases}
\]
Let $E'_\ell$ be the $A$-module generated by  $x_1^{-i_1}\cdots x_n^{-i_n}$, 
where $i_1+\cdots+i_n\leq\ell+n$. For instance, 
$E'_0=Rx_1^{-1}\cdots x_n^{-1}=\kappa x_1^{-1}\cdots x_n^{-1}$ can be identified 
with $\kappa$. Clearly, $\m E'_{\ell+1}=E'_\ell$. In general, 
$\m^j E'_{\ell+j}=E'_\ell$ for $j,\ell\geq 0$. We may assign $i_1+\cdots+i_n-n$ to be the order of
$x_1^{-i_1}\cdots x_n^{-i_n}$ in compliant with the identification of the coefficient field $\kappa$ with $E'_0$, where scalars have order zero.
\begin{prop}\label{prop:9yyyq4}
Let $A$ be the power series ring $\kappa[\![x_1,\ldots,x_n]\!]$ over a field $\kappa$. 
The $A$-module $\Hom_\kappa^c(A,\kappa)$ of inverse polynomials is isomorphic to the 
solution module $E(\kappa)$.
\end{prop}
\begin{proof} 
We use induction to prove that $E'_\ell$ satisfies the initial condition for $\ell>0$.
It is straightforward to check that any $\m$-system in $\m E'_1=E'_0$ has a 
solution in $E'_1$. In other words, $E'_1$ satisfies the initial condition. Assuming 
that $E'_\ell$ satisfies the initial condition for some number $\ell>0$, we aim to show that $E'_{\ell+1}$ does as well.

The kernel of the composition maps
\[
\Hom_A(\m,E'_\ell)\to\Hom_A(\m^{\ell+1},\m^\ell E'_\ell)\xrightarrow{\sim}\Hom_A(\m^{\ell+1},\kappa)
\]
is $\Hom_A(\m,E'_{\ell-1})$. The kernel of the composition maps
\[
E'_{\ell+1}\to\Hom_A(\m,E'_\ell)\to\Hom_A(\m,E'_\ell)/\Hom_A(\m,E'_{\ell-1})
\]
is $E'_\ell$. So we may regard $E'_{\ell+1}/E'_\ell$ and $\Hom_A(\m,E'_\ell)/\Hom_A(\m,E'_{\ell-1})$ as
$\kappa$-vector subspaces of $\Hom_A(\m^{\ell+1},\kappa)=\Hom_\kappa(\m^{\ell+1}/\m^{\ell+2},\kappa)$.
The dimension of $E'_{\ell+1}/E'_\ell$ is $\binom{\ell+n}{\ell+1}$. 
Indeed, a basis of $E'_{\ell+1}/E'_\ell$ is given by inverse monomials 
$x_1^{-i_1}\cdots x_n^{-i_n}$, where $i_1+\cdots+i_n=\ell+n+1$.
The dimension of $\Hom_\kappa(\m^{\ell+1}/\m^{\ell+2},\kappa)$ is $\binom{\ell+n}{\ell+1}$ as well. 
Indeed, a dual basis of $\m^{\ell+1}/\m^{\ell+2}$ is given by the monomials 
$x_1^{i_1}\cdots x_n^{i_n}$, where $i_1+\cdots+i_n=\ell+1$. Therefore
\begin{equation}\label{eq:71788}
E'_{\ell+1}/E'_\ell=\Hom_A(\m,E'_\ell)/\Hom_A(\m,E'_{\ell-1})=
\Hom_A(\m^{\ell+1},\kappa).
\end{equation}
Parallel to (\ref{eq:781677}), we have an exact sequence
\[
0\to E'_\ell\to E'_{\ell+1}\to\Hom_A(\m,E'_\ell)/\Hom_A(\m,E'_{\ell-1})\to 0.
\]
Given an $\m$-system $\varphi$ in $\m E'_{\ell+1}=E'_\ell$, there is an element 
$\alpha_1\in E'_{\ell+1}$ such that $a\mapsto\varphi(a)-a\alpha_1$ gives rise to an 
$\m$-system in $E'_{\ell-1}=\m E'_\ell$. The induction hypothesis provides a solution
$\alpha_0\in E'_\ell$ to the latter system. The element $\alpha_1+\alpha_0\in E'_{\ell+1}$
is a solution to $\varphi$. Therefore $E'_{\ell+1}$ satisfies the initial condition.

With the notation in Definition~\ref{def:76167}, we proceed to construct the solution 
module of $M:=E'_\ell$ up to the first order. The free module $F$ has an $A$-basis 
corresponding to the $\kappa$-basis of $E'_{\ell+1}/E'_\ell$ provided by
$x_1^{-i_1}\cdots x_n^{-i_n}$, where $i_1+\cdots+i_n=\ell+n+1$. While $M$ is
generated by those $x_1^{-i_1}\cdots x_n^{-i_n}$ with $i_1+\cdots+i_n<\ell+n+1$,
one checks directly that $(F\oplus M)/N\simeq E'_{\ell+1}$ as $A$-modules. In other
words, $E'_{\ell+1}\simeq E_1(E'_\ell)$ for $\ell\geq 0$. Both $E'_\ell$ and $E_\ell(\kappa)$
can be constructed iteratively starting from $\kappa$ as solution modules up to increasing orders. Therefore $\Hom_\kappa^c(A,\kappa)\simeq E(\kappa)$.
\end{proof}

Local cohomology provides another source of linear closures. Let $A$ be a
Gorenstein local ring of dimension $n$. The $n$-th 
local cohomology module $\mH^n_\m(A)$ supported in the maximal ideal $\m$ is a linear closure of the residue field of $A$. See \cite[(3.8)]{hu:pmzds} for another view in the relative situation. If $A$ is a power series ring
$\kappa[\![x_1,\ldots,x_n]\!]$ over a field $\kappa$, an isomorphism
$\mH^n_\m(A)\simeq\Hom_\kappa^c(A,\kappa)$ given by residues interprets local duality. See \cite[(5.9)]{hu:pmzds} for a general statement.

\end{document}